\UseRawInputEncoding
\documentclass[10pt]{article}
\usepackage{relsize}

\usepackage[dvips]{color}
\usepackage{epsfig}
\usepackage{float,amsthm,amssymb,amsfonts}
\usepackage{ amssymb,amsmath,graphicx, amsfonts, latexsym}
\usepackage{xcolor}
\begin{document}
\theoremstyle{plain}
\newtheorem{theorem}{{\bf Theorem}}[section]
\newtheorem{corollary}[theorem]{Corollary}
\newtheorem{lemma}[theorem]{Lemma}
\newtheorem{proposition}[theorem]{Proposition}
\newtheorem{remark}[theorem]{Remark}

\theoremstyle{definition}
\newtheorem{defn}{Definition}
\newtheorem{definition}[theorem]{Definition}
\newtheorem{example}[theorem]{Example}
\newtheorem{conjecture}[theorem]{Conjecture}

\def\im{\mathop{\rm Im}\nolimits}
\def\dom{\mathop{\rm Dom}\nolimits}
\def\rank{\mathop{\rm rank}\nolimits}
\def\nullset{\mbox{\O}}
\def\ker{\mathop{\rm ker}\nolimits}
\def\implies{\; \Longrightarrow \;}

\def\GR{{\cal R}}
\def\GL{{\cal L}}
\def\GH{{\cal H}}
\def\GD{{\cal D}}
\def\GJ{{\cal J}}

\def\set#1{\{ #1\} }
\def\z{\set{0}}
\def\Sing{{\rm Sing}_n}
\def\nullset{\mbox{\O}}

\title{On  the small Schr\"{o}der semigroup $\mathcal{SS}^{\prime}_{n}$}
\author{\bf  Muhammad Mansur Zubairu\footnote{Corresponding Author. ~~Email: \emph{mmzubairu.mth@buk.edu.ng}}, Abdullahi Umar and  Fatma Salim Al-Kharousi   \\
\it\small  Department of Mathematics, Bayero  University Kano, P. M. B. 3011, Kano, Nigeria\\
\it\small  \texttt{mmzubairu.mth@buk.edu.ng}\\[3mm]
\it\small Department of Mathematical Sciences,\\
\it\small Khalifa University, P. O. Box 127788, Sas al Nakhl, Abu Dhabi, UAE\\
\it\small  \texttt{abdullahi.umar@ku.ac.ae}\\[3mm]
\it\small  Department of Mathematics,\\
\it\small College of Science,\\
\it\small Sultan Qaboos University. \\
\it\small \texttt{fatma9@squ.edu.om} }
\date{\today}
\maketitle\

\begin{abstract}  Let $[n]$ be a  finite $n-$chain $\{1, 2, \dots, n\}$, and let $\mathcal{LS}_{n}$ be the Schr\"{o}der monoid, consisting of all isotone and order-decreasing partial transformations on $[n]$. Furthermore, let $\mathcal{SS}^{\prime}_{n} = \{\alpha \in \mathcal{LS}_{n} : \, 1\not\in \textnormal{Dom } \alpha\}$ be the subsemigroup of $\mathcal{LS}_{n}$, consisting of all transformations in $\mathcal{LS}_{n}$, each of whose domain does not contain $1$. For $1 \leq p \leq n$, let $K(n,p) = \{\alpha \in \mathcal{SS}^{\prime}_{n} : \, |\im \, \alpha| \leq  p\}$
 be the two-sided ideal of $\mathcal{SS}^{\prime}_{n}$. Moreover, let  ${RSS}^{\prime}_{n}(p)$ denote the Rees quotient of $K(n,p)$.  It is shown in this article that for any $S$ in $\{\mathcal{SS}^{\prime}_{n}, K(n,p), {RSS}^{\prime}_{n}(p)\}$, $S$ is right abundant for all values of $n$, but not left abundant for all $n \geq 2$. In addition, the rank of the Rees quotient ${RSS}^{\prime}_{n}(p)$ is shown to be equal to the rank of the two-sided ideal $K(n,p)$, which is equal to $\binom{n-1}{p-1}+\sum\limits_{k=p}^{n-1}\binom{n-1}{k} \binom{k-1}{p-1}$. Finally, the rank of $\mathcal{SS}^{\prime}_{n}$ is determined to be $3n-4$. \end{abstract}


\section{Introduction and Preliminaries}
  Denote $[n]$ to be the finite chain $\{1,2, \dots ,n\}$. A mapping $\alpha$ whose domain and range are both subsets of $[n]$   is known as a \emph{partial} \emph{transformation} of $[n]$, and it is said be  \emph{full} (or \emph{total}) if such a map has the whole of $[n]$ as its domain.  The  collection of all partial transformations on $[n]$ is usually denoted by $\mathcal{P}_{n}$, and is referred to as \emph{the semigroup of all partial transformations} on $[n]$, more commonly known as \emph{the partial symmetric monoid}. A transformation $\alpha\in \mathcal{P}_{n}$ is said to be an \emph{ isotone} map (resp., an \emph{anti-tone} map) if  (for all $x,y \in \dom\,\alpha$) $x\leq y$ implies $x\alpha\leq y\alpha$ (resp., $x\alpha\geq y\alpha$); \emph{order decreasing} if (for all $x\in \dom\,\alpha$) $x\alpha\leq x$.  As in \cite{auc}, we shall refer to $\mathcal{LS}_{n}$ (\emph{semigroup of all isotone order-decreasing partial transformations} on $[n]$) as the \emph{large} \emph{Schr\"{o}der} monoid and  it is defined as: \begin{equation}\label{qn111}\mathcal{LS}_{n}= \mathcal{OP}_n\cap \mathcal{DP}_n ,\end{equation} \noindent where $\mathcal{DP}_n$ and $\mathcal{OP}_n$  denote  \emph{the semigroup of all order-decreasing partial transformations} on $[n]$ and the \emph{semigroup of all isotone partial transformations} on $[n]$, respectively. Moreover, let

\begin{equation}\label{qn1}
 	\mathcal{SS}_{n} = \{\alpha \in \mathcal{LS}_{n} : 1 \in \textnormal{Dom } \alpha \}
 \end{equation}
\noindent be the set of all maps in $\mathcal{LS}_{n}$ each of whose domain contains  1 and
\begin{equation}\label{qn2}
	\mathcal{SS}^{\prime}_{n} = \{\alpha \in \mathcal{LS}_{n} : 1 \notin \text{Dom } \alpha\}
\end{equation}
 \noindent be the set of all maps in $\mathcal{LS}_{n}$ each of whose domain does not contain  1. In other words, $\mathcal{SS}^{\prime}_{n}$ is the set complement of $\mathcal{SS}_{n}$. These sets first appeared in \cite{al5}, where they were shown to be subsemigroups of $\mathcal{LS}_{n}$ and  to have equal size, which coincides with the (\emph{small}) \emph{Schr\"{o}der number}:
\[s_{0}=1, \, \, s_{n}= \frac{1}{2(n+1)} \sum\limits_{r=0}^{n}\binom{n+1}{n-r}\binom{n+r}{r} \, \, (n\geq 1).\]
Subsequently, in \cite{zuf}, the monoids $\mathcal{LS}_{n}$ and $\mathcal{SS}_{n}$ are both shown to be idempotent-generated abundant semigroups \cite{FOUN2}. Moreover, in \cite{zuf}, the ranks of their respective two-sided ideals and their Rees quotient semigroups were all obtained.

However, it seems the algebraic as well as the rank properties of the semigroup $\mathcal{SS}^{\prime}_{n}$ have not been investigated. In this paper, we investigate certain algebraic properties of $\mathcal{SS}^{\prime}_{n}$ and its rank properties. Therefore, this paper is a natural sequel to \cite{al5, zuf}. It is worth noting that $\mathcal{SS}_{1}^{\prime}=\emptyset$, as such we shall henceforth consider $n\geq 2$. Also, it is worth noting that $\mathcal{SS}_{n}^{\prime}$ has no identity element; therefore, in line with \cite{al5, zuf}, we shall refer to it as the \emph{small} \emph{Schr\"{o}der semigroup}.

Moreover, we shall adopt the right-hand composition of two elements say $\alpha$ and $\beta$ in $\mathcal{P}_{n}$ defined as $x(\alpha\circ\beta)=((x)\alpha)\beta$ for all $x\in\dom\, \alpha$. We may without ambiguity be using the notation $\alpha\beta$ to denote $\alpha\circ\beta$. The notations $1_{[n]}$, $\im \alpha$, $\dom \alpha$, $h(\alpha)=|\im \, \alpha|$ and $F(\alpha)=\{x\in \dom \, \alpha: \, x\alpha=x\}$ shall denote the identity map on $[n]$, the image set of a map $\alpha$, the domain set of the map $\alpha$, the height of $\alpha$, and the set of fixed points of $\alpha$, respectively. Also, we will let $f(\alpha)=|F(\alpha)|$ to be the number of fixed points of $\alpha$.

  Furthermore, for $0\le p\le n-1$, we let \begin{equation} \label{kn} K(n,p)=\{\alpha\in  \mathcal{SS}_{n}^{\prime}: \, |\im \, \alpha|\le p\}\end{equation}
  \noindent be the two sided ideal of $\mathcal{SS}_{n}^{\prime}$, consisting of all decreasing isotone maps in  $\mathcal{SS}_{n}^{\prime}$, each with a height at most $p$.

  Furthermore, for $p\geq 1$, we let \begin{equation}\label{knn} {RSS}_{n}^{\prime}(p)= K(n,p)/ K(n, p-1)  \end{equation}
  \noindent be the Rees quotient semigroup of $K(n,p)$.  The elements of ${RSS}_{n}^{\prime}(p)$ can be considered as the elements of $\mathcal{SS}_{n}^{\prime}$ of exactly height $p$. The product of two elements of  ${RSS}_{n}^{\prime}(p)$ is $0$ if their product in ${RSS}_{n}^{\prime}(p)$ has a height strictly less than $p$, otherwise it is in ${RSS}_{n}^{\prime}(p)\setminus\{0\}$.

\indent   As in \cite{HRS}, every $\alpha\in \mathcal{SS}_{n}^{\prime}$  can be represented as
\begin{equation}\label{1}\alpha=\begin{pmatrix}A_1&\cdots&A_p\\a_1&\cdots&a_p\end{pmatrix}   \,  (1\le p\le n-1),\end{equation}  where  $a_{i}\leq \min A_{i}$ for all $1\leq i\leq p$ since $\alpha$ is  a decreasing map, and  each of the sets $A_i$ $(1\le i\le p)$  denote an equivalence class defined by the relation $\textnormal{ker }\alpha=\{(x, y)\in \dom \, \alpha\times \dom \, \alpha: \,  x\alpha=y\alpha\}$. The collection of these equivalence classes  shall be denoted as $\textnormal{\bf Ker }\alpha=\{A_1, A_2, \dots, A_p\}$.  Furthermore, $\textnormal{\bf Ker }\alpha$ is linearly ordered (i.e., for $i<j$, $A_{i}<A_{j}$ if and only if $a<b$ for all $a\in A_{i}$ and $b\in A_{j}$). Moreover, we may without loss of generality assume that $1\leq a_{1}<a_{2}<\dots<a_{p}\leq n$, since $\alpha$ is an isotone map.
For more details on basic concepts in semigroup theory, we recommend to the reader  the books of  Howie \cite{howi} and Higgins \cite{ph}.

At this juncture, we briefly describe the general structure of the paper. In Section 1, we provide definitions of some basic terms. In Section 2, we characterize Green's relations as well as their starred analogues, demonstrating that the small Schrödler semigroup \( \mathcal{SS}^{\prime}_{n} \) and its two-sided ideal \( K(n,p) \) are right abundant but not left abundant for all \( n > 2 \). Moreover, we compute the number of \( \mathcal{R}^* \)-classes, \( \mathcal{L}^{*} \)-classes, and the number of idempotent elements in \( \mathcal{SS}^{\prime}_{n} \). Finally, in Section 3, we calculate the ranks of the Rees quotients, the two-sided ideals \( K(n,p) \), and the small Schrödler semigroup \( \mathcal{SS}^{\prime}_{n} \).

\section{Regularity, Green's relations and starred Green's relations on the Schr\"{o}der semigroup  $\mathcal{SS}^{\prime}_{n} $}

 In semigroup theory, there are five Green's relations, namely $\mathcal{L,R,D , J\ \text{and } H}$. The relation $\mathcal{L}$ is defined as: for $a$, $b$ in a semigroup $S$, $(a,b)\in \mathcal{L}$ if and only if  $a$ and $b$ generates the same principal left ideal; while the relation $\mathcal{R}$ is defined dually; the relation $\mathcal{J}$ is defined as:$(a,b)\in \mathcal{J}$ if and only if  $a$ and $b$ generates the same principal ideal;  the relation $\mathcal{D}$ is the join of $\mathcal{L}$ and $\mathcal{R}$; while $\mathcal{H}=\mathcal{L} \cap \mathcal{R}$. It is well-known that  in a finite semigroup $\mathcal{D }=\mathcal{J}$ (see [\cite{howi}, Proposition 2.1.4]). Thus, we shall only focus on characterizing the relations $\mathcal{L,R,D \, \text{and } H}$ on the  small  Schr\"{o}der semigroup $\mathcal{SS}_{n}^{\prime}$. An element $a$ in a semigroup $ S$  is said to be \emph{regular} if  $a=aba$  for some $b\in S$, and $S$ is said to be a \emph{regular semigroup} if all its elements are regular. The structural properties of a regular semigroup has been extensively studied in the literature, see for example section 2 and 4 of Howie's book \cite{howi}.

From this point forward in this section, we shall  refer to $\alpha$ and $\beta$ in $\mathcal{SS}^{\prime}_{n}$ as:
 \begin{equation} \label{eqq3}
	\alpha = \begin{pmatrix}A_1&\cdots& A_p\\a_{1}&\cdots& a_p\end{pmatrix} \text{and} \  \beta = \begin{pmatrix} B_1 &  \cdots & B_p \\ b_{1} & \cdots  & b_p \end{pmatrix}  \, (1\leq p\leq n-1).
\end{equation}

 We now have the following theorems.

 \begin{theorem}\label{l}
 Let $\alpha,\beta \in \mathcal{SS}^{\prime}_{n}$  be as in   \eqref{eqq3}. Then $\alpha\mathcal{L}\beta$ if and only if
 $\im \, \alpha=\im \, \beta$ \emph{(}i.e., $a_i = b_i$ for $1\leq i\leq p$\emph{)} and $\min A_i =  \min B_i$ for all $1\leq i\leq p$.
 \end{theorem}

 \begin{proof}
	 The proof  resembles the proof of [\cite{umar}, Lemma 2. 2. 1(2)].	
 \end{proof}
\begin{theorem}\label{r}
 Let $\mathcal{SS}_{n}^{\prime}$ be as defined in \eqref{qn2}. Then $\mathcal{SS}_{n}^{\prime}$ is $\mathcal{R}-$trivial.
 \end{theorem}
 \begin{proof} It is clear that   $\mathcal{SS}_{n}^{\prime}$ is  $\mathcal{R}-$trivial being a subsemigroup of an $\mathcal{R}-$trivial semigroup, $\mathcal{LS}_{n}$ [\cite{ph1}, Theorem 4.2].
 \end{proof}

  As a consequence of the above Theorems \ref{l} and \ref{r}, we readily have the following corollaries.
\begin{corollary}
On the Schr\"{o}der semigroup $\mathcal{SS}_{n}^{\prime}$, $\mathcal{H} = \mathcal{R}$.
\end{corollary}

 \begin{corollary}\label{rem1} Let $\alpha \in  \mathcal{SS}_{n}^{\prime}$. Then $\alpha$ is regular if and only if $\alpha$ is an idempotent. Consequently, the semigroup $\mathcal{SS}_{n}^{\prime}$ is nonregular.
 \end{corollary}
\begin{proof} The result follows from the fact that in an  $\mathcal{R}$-trivial semigroup, every nonidempotent element is not regular.
 \end{proof}

 \begin{theorem} Let $\mathcal{SS}_{n}^{\prime}$ be as defined in \eqref{qn2}. Then $ \mathcal{D} = \mathcal{L}$.	
	\end{theorem}
\begin{proof}
The result follows from the fact that $\mathcal{L}\subseteq\mathcal{D}$, and that $\mathcal{SS}_{n}^{\prime}$ is $\mathcal{R}$-trivial from Theorem \ref{r}.
\end{proof}

Consequently, based on the above three theorems, we deduce the following characterization of Green's equivalences on a semigroup $S$ in $\{{RSS}_{n}^{\prime}(p),  \, K(n,p) \}$.

 \begin{theorem} Let $S\in \{{RSS}_{n}^{\prime}(p), \, K(n,p) \}$ and  let $\alpha, \, \beta \in S$.
 Then \begin{itemize} \item[(i)] $\alpha \mathcal{L} \beta$ if and only if $\im \, \alpha = \im \, \beta$ \emph{(}i.e., $a_i = b_i$ for $1 \leq i \leq p$\emph{) }and $\min A_i = \min B_i$ for all $1 \leq i \leq p$; \item[(ii)] $S$ is $\mathcal{R}$-trivial; \item[(iii)] $\mathcal{H} = \mathcal{R}$;  \item[(iv)] $\mathcal{D} = \mathcal{L}$.\end{itemize} Hence, for $p \geq 2$, the semigroup $S$ is nonregular.
\end{theorem}

If a semigroup lacks regularity, it is customary to investigate the starred Green's relations to determine its algebraic classification. Therefore, we will proceed to characterize the starred counterparts of Green's equivalences on $S\in \{\mathcal{SS}_{n}^{\prime}, {RSS}_{n}^{\prime}(p), \, K(n,p) \}$. We refer the reader  to Fountain \cite{FOUN2} for  the definitions and basic properties of these relations.

The five  starred Green's equivalences are: $\mathcal{L}^*$, $\mathcal{R}^*$, $\mathcal{D}^*$, $\mathcal{J}^*$, and $\mathcal{H}^*$ . The relation $\mathcal{H}^*$ is the intersection of $\mathcal{L}^*$ and $\mathcal{R}^*$, while $\mathcal{D}^*$ is the join of $\mathcal{L}^*$ and $\mathcal{R}^*$. It is a known fact that in a finite non-regular semigroup, $\mathcal{L^*}\circ\mathcal{R^*}$ does not necessarily commute. Moreover, the relations $\mathcal{L^*}$ and $\mathcal{R}^*$ have the following characterizations on any semigroup $S$:
 \begin{equation}\label{l8}
	\mathcal{L}^* = \{(a,b)  \in S\times S \ : \ (\text{for all} \ x,y \in S^1) \  ax = ay  \iff bx = by \};
\end{equation}

\begin{equation}\label{r8}
	\mathcal{R}^* = \{(a,b)  \in S\times S \ : \  (\text{for all } x,y \in S^1) \  xa = ya  \iff xb = yb \}.
\end{equation}

 A semigroup $S$ is said to be \emph{left abundant} if each $\mathcal{L}^*$-class contains an idempotent; it is said to be \emph{right abundant} if each $\mathcal{R}^*$-class contains an idempotent; and it is \emph{abundant} if it is both left and right abundant. These classes of semigroups were introduced by Fountain \cite{FOUN, FOUN2}. There are many classes of  transformation semigroups that have been shown to be either left abundant, right abundant, or abundant; see, for example \cite{al1, um,umar,  quasi,  zm1}. Before we characterize the starred Green's relations, we need the following definition and halfway results of lemmas from \cite{quasi}: A subsemigroup $U$ of $S$ is called  a \emph{right inverse ideal} of $S$ if for all $u \in U$, there exists $u^{\prime} \in S$ such that $uu^{\prime}u = u$ and  $uu^{\prime}\in U$; it is called  a \emph{left inverse ideal} of $S$ if for all $u \in U$, there exists $u^{\prime} \in S$ such that $uu^{\prime}u = u$ and  $u^{\prime}u\in U$; and it is an  \emph{inverse ideal} of $S$ if it is both left and right inverse ideal of $S$.

 \begin{lemma}[\cite{quasi}, Lemma 3.1.8.]\label{inv1}  Every right inverse  ideal $U$ of a semigroup $S$ is right abundant.
 \end{lemma}

 \begin{lemma} [\cite{quasi}, Lemma 3.1.9.]  \label{inv2}  Let $U$ be a right inverse ideal of a semigroup $S$. Then $\mathcal{R}^{*}( U) = \mathcal{R}(S) \cap(U \times U)$.
 \end{lemma}

The next result is now at our disposal.
 \begin{theorem}\label{inv} Let \(\mathcal{SS}^{\prime}_{n}\)  be as defined in \eqref{qn2}. Then, for $n\ge 2$, \(\mathcal{SS}^{\prime}_{n}\) is a right inverse ideal of  $\mathcal{P}_{n}$ but not left.
 \end{theorem}
 \begin{proof} Let $\alpha\in \mathcal{SS}_{n}^{\prime}$ be as expressed in \eqref{1}, and let $t_{i}=\min A_{i}$ for all $1\leq i\leq p$. Now define $\alpha^{\prime}$ as: \[\alpha^{\prime}=\begin{pmatrix}
a_1  & \cdots & a_p\\
t_1   & \cdots & t_p
\end{pmatrix} .\]
\noindent Clearly, $\alpha^{\prime}$ is in $\mathcal{P}_{n}$. Notice that:

\begin{align*}\alpha\alpha^{\prime}\alpha &=\begin{pmatrix}
A_1  & \cdots & A_p\\
a_1   & \cdots & a_p
\end{pmatrix}\begin{pmatrix}
a_1  & \cdots & a_p\\
t_1   & \cdots & t_p
\end{pmatrix}\begin{pmatrix}
A_1  & \cdots & A_p\\
a_1   & \cdots & a_p
\end{pmatrix}\\&= \begin{pmatrix}
A_1  & \cdots & A_p\\
a_1   & \cdots & a_p
\end{pmatrix}=\alpha. \end{align*}

Moreover, \[\alpha\alpha^{\prime}=\begin{pmatrix}
A_1  & \cdots & A_p\\
a_1   & \cdots & a_p
\end{pmatrix}\begin{pmatrix}
a_1  & \cdots & a_p\\
t_1   & \cdots & t_p
\end{pmatrix}
=\begin{pmatrix}
A_1  & \cdots & A_p\\
t_1   & \cdots & t_p
\end{pmatrix}\in E(\mathcal{SS}_{n}^{\prime})\subset \mathcal{SS}^{\prime}_{n}.\]

\noindent However, \[\alpha^{\prime}\alpha=\begin{pmatrix}
a_1  & \cdots & a_p\\
t_1   & \cdots & t_p
\end{pmatrix}\begin{pmatrix}
A_1  & \cdots & A_p\\
a_1   & \cdots & a_p
\end{pmatrix}=\begin{pmatrix}
a_1  & \cdots & a_p\\
a_1   & \cdots & a_p
\end{pmatrix}=\text{1}_{\im \, \alpha},\]\noindent which is not necessary in \(\mathcal{SS}^{\prime}_{n}\), for example its possible $a_{1}=1$.

  Thus,  $\mathcal{SS}_{n}^{\prime}$ is a right inverse ideal of $\mathcal{P}_{n}$, but not left.
 \end{proof}

 Consequently, we have the following result.

\begin{theorem}
	Let \(\mathcal{SS}^{\prime}_{n}\)  be as defined in \eqref{qn2}. Then for $n\ge 2$, \(\mathcal{SS}^{\prime}_{n}\) is  right abundant.
\end{theorem}

\begin{proof}
	 The result follows from Theorem \ref{inv}  and Lemma \ref{inv1}.
\end{proof}

 \begin{theorem} \label{a1}
Let \(\mathcal{SS}^{\prime}_{n}\)  be as defined in \eqref{qn2}, then  for $\alpha, \beta\in S$ we have:
 \begin{itemize}
   \item[(i)] $\alpha\mathcal{L}^*\beta$  if and only if $\im  \alpha = \im  \beta$;
   \item[(ii)] $\alpha\mathcal{R}^*\beta$ if and only if $\ker  \alpha = \ker  \beta$;
   \item[(iii)] $\alpha\mathcal{H}^*\beta$ if and only if $\alpha=\beta$;
   \item[(iv)] $\alpha\mathcal{D}^*\beta$ if and only if $|\im  \alpha| = |\im   \beta|$.
 \end{itemize}
 \end{theorem}

\begin{proof}
\begin{itemize} \item[(i)] Let $\alpha, \beta \in \mathcal{SS}^{\prime}_{n}$  be as expressed in \eqref{eqq3}  such that $\alpha\mathcal{L^*}\beta$ and $\im  \alpha = \{a_1,a_2,\dots,a_p\}$. Further, let $\gamma_{1}$ = $\begin{pmatrix}
{1} &  \cdots & n\\
1   & \cdots & n
\end{pmatrix}$.  Then clearly $\gamma_{1}=1_{[n]}\in (\mathcal{SS}^{\prime}_{n})^{1}= \mathcal{SS}^{\prime}_{n}\cup \{1_{[n]}\}$ and by \eqref{l8}, we see that:

$$\alpha\gamma_{1} = \alpha 1_{[n]}
\iff \beta\gamma_{1} = \beta1_{[n]}$$
\noindent which implies that  $\im \, \beta \subseteq \{a_1,a_2, \dots,a_p\} = \im \, \alpha$ or  $\im \, \beta = \varnothing \subseteq \{a_1,\dots, a_p\} = \im \, \beta$,  i.e.,  $\im \, \beta \subseteq \im \, \alpha$.

 Similarly, let $\im \, \beta = \{b_1,b_2, \dots,b_p\}$   and define   $\gamma_{2} = \begin{pmatrix}
	1 &  \cdots & n\\
	1 &  \cdots & n
\end{pmatrix}=1_{[n]}$. Then clearly $\gamma_{2} \in (\mathcal{SS}^{\prime}_{n})^{1}$ and by  \eqref{l8}, we see that:

$$\alpha\gamma_{2} = \alpha 1_{[n]}
\iff \beta\gamma_{2} = \beta1_{[n]}$$
\noindent which implies that  $\im \, \alpha \subseteq \{b_1,b_2, \dots,b_p\} = \im \, \beta$  or $\im \, \alpha = \varnothing \subseteq \{b_1,b_2,\dots, b_p\} = \im \, \beta$. i.e.,  $\im \, \alpha \subseteq \im \, \beta$, as required.

Conversely, suppose $\im  \, \alpha = \im \, \beta.$ Then by [\cite{howi}, Exercise 2.6.17] $\alpha \mathcal{L}^{\mathcal{P}_n} \beta$, and it follows from definition that $\alpha \mathcal{L^*} \beta$.

\item[(ii)] The result follows from Theorem \ref{inv}, Lemma \ref{inv2} and [\cite{howi}, Exercise 2.6.17], while (iii) follows from (i) and (ii) and the fact that $\alpha$ and $\beta$ are both isotone maps.
\item[(iv)] Suppose $\alpha\mathcal{D}^{*}\beta$. Then by (\cite{howi}, Proposition 1.5.11), it means that there exist elements $\delta_{1},~\delta_{2}, \dots,~\delta_{2n-1}$ in $\mathcal{SS}^{\prime}_{n}$ such that $\alpha\mathcal{L}^{*}\delta_{1}$, $\delta_{1}\mathcal{R}^{*}\delta_{2}$, $\delta_{2}\mathcal{L}^{*}\delta_{3},\dots,$ $\delta_{2n-1}\mathcal{R}^{*}\beta$ for some $n\in ~ \mathbb{{N}}$. Consequently, from  (i) and (ii), we can deduce that $\im~\alpha=\im~\delta_{1}$, ${\ker}~\delta_{1}={\ker}~\delta_{2}$, $\im~\delta_{2}=\im~\delta_{3},\dots,$ $\ker~\delta_{2n-1}=\ker~\beta$. Now it follows that $|\im~\alpha|=|\im~\delta_{1}|=|\dom~\delta_{1}/ \ker~\delta_{1}|=|\dom~\delta_{2}/ \ker~\delta_{2}|=\dots=|\dom~\delta_{2n-1}/ \ker~\delta_{2n-1}|=|\dom~\beta/ \ker~\beta|=|\im~\beta|.$

Conversely, suppose that $|\im~\alpha|=|\im~\beta|$ where \begin{equation*}\label{2} \alpha=\left(\begin{array}{ccc}
                                                                            A_{1}  & \cdots & A_{p} \\
                                                                            a_{1} & \cdots & a_{p}
                                                                          \end{array}
\right)\text{ and } \beta=\left(\begin{array}{ccc}
                                                                            B_{1}  & \cdots & B_{p} \\
                                                                            b_{1} & \cdots & b_{p}
                                                                          \end{array}
\right).\end{equation*}

Now define \begin{equation*}\label{2} \delta=\left(\begin{array}{ccc}
                                                                            A_{1}  & \cdots & A_{p} \\
                                                                            {1} & \cdots & {p}
                                                                          \end{array}
\right)\text{ and } \gamma=\left(\begin{array}{ccc}
                                                                            B_{1}  & \cdots & B_{p} \\
                                                                            {1} & \cdots & {p}
                                                                          \end{array}
\right).\end{equation*}

\noindent Clearly, $\delta$ and $\gamma$ are in $\mathcal{SS}^{\prime}_{n}$. Notice that $\ker \, \alpha= \ker \, \delta$, $\im \, \delta=\im \, \gamma$ and $\ker \, \gamma=\ker \, \beta$. Thus by (i) and (ii) we see that $\alpha \mathcal{R}^{*} \delta \mathcal{L}^{*} \gamma \mathcal{R}^{*} \beta$.

 \noindent Similarly,  define $\delta=\left(\begin{array}{ccc}
                                                                            n-p+{1}  & \cdots & n \\
                                                                            a_{1} & \cdots & a_{p}
                                                                          \end{array}
\right)$ and  $\gamma=\left(\begin{array}{ccc}
                                                                            n-p+1  & \cdots & n \\
                                                                            b_{1} & \cdots & b_{p}
                                                                          \end{array}
\right)$. Clearly, $\delta$ and $\gamma\in \mathcal{SS}^{\prime}_{n}$. Moreover,  notice that $\im \, \alpha=\im  \, \delta$,   $\ker \, \delta= \ker \, \gamma$,  $\im \, \gamma=\im \,  \beta$.

 Thus by (i) and (ii) we have $\alpha \mathcal{L}^{*} \delta \mathcal{R}^{*} \gamma \mathcal{L}^{*}\beta$. Hence, by (\cite{howi}, Proposition 1.5.11) it follows that $\alpha\mathcal{D}^{*}\beta$.  The proof is now complete.
\end{itemize}
\end{proof}

\begin{lemma}\label{uaaaa} On the Schr\"{o}der semigroup  $\mathcal{SS}^{\prime}_{n}$  \emph{(}$n\geq 4$\emph{)}, we have $\mathcal{D}^{*}=\mathcal{R}^{*}\circ\mathcal{L}^{*}\circ\mathcal{R}^{*}=\mathcal{L}^{*}\circ\mathcal{R}^{*}\circ\mathcal{L}^{*}$ .
\end{lemma}
\begin{proof} The sufficiency  follows from the converse of the  proof of (iv) in the above theorem, while for the necessity, we have to prove that   $\mathcal{L}^{*}\circ\mathcal{R}^{*}\neq \mathcal{R}^{*}\circ\mathcal{L}^{*}$. Take \[\alpha=\left(\begin{array}{cc}
                                                                            2  &  3 \\
                                                                            {2} &3
                                                                          \end{array}
\right) \text{ and } \beta=\left(\begin{array}{cc}
                                                                            2  &  4 \\
                                                                            {2} &4
                                                                          \end{array}
\right). \]

\noindent Now define $\delta=\left(\begin{array}{cc}
                                                                            2  &  4 \\
                                                                            {2} &3
                                                                          \end{array}
\right).$ Then clearly $\im \, \alpha=\im \, \delta$ and $\dom \, \delta=\dom \, \beta$, and so  $\alpha \mathcal{L}^{*} \delta \mathcal{R}^{*}\beta$. i.e., $(\alpha, \beta)\in \mathcal{L}^{*} \circ \mathcal{R}^{*}$.

On the other hand if we have $(\alpha, \beta)\in  \mathcal{R}^{*} \circ \mathcal{L}^{*}$, then it means there must exist $\gamma \in\mathcal{SS}^{\prime}_{n}$ such that $\alpha \mathcal{R}^{*} \gamma \mathcal{L}^{*}\beta$. However, this means that $\dom \, \alpha= \dom \, \gamma=\{2,3\}$ and $\im \, \gamma=\im \, \beta=\{2,4\}$, which is impossible. The result now follows.
\end{proof}

\begin{lemma}\label{uaaa} On the semigroups  ${RSS}_{n}^{\prime}(p)$ and  $K(n,p)$ ($1\leq p\leq n-1$), we have $\mathcal{D}^{*}=\mathcal{R}^{*}\circ\mathcal{L}^{*}\circ\mathcal{R}^{*}=\mathcal{L}^{*}\circ\mathcal{R}^{*}\circ\mathcal{L}^{*}.$
\end{lemma}
\begin{proof} The proof is the same as  the proof of  the above lemma.
\end{proof}
\begin{lemma}\label{qna2}
	The Schr\"{o}der semigroup  $\mathcal{SS}^{\prime}_{n}$ is not left abundant for all $n \ge 2$.
\end{lemma}

\begin{proof}
	Consider $\alpha = \begin{pmatrix}
		2\\1
	\end{pmatrix}  \in \mathcal{SS}^{\prime}_{n} \ \text{for all } n \ge 2$. Now the $\mathcal{L}^* - \, class$ of $\alpha$ is:
$${L^*}_{\alpha} = \Bigg\{
\begin{pmatrix}
	2\\1
\end{pmatrix}, \begin{pmatrix}
\{2,3\} \\ 1
\end{pmatrix}, \begin{pmatrix}
\{2,3,4\} \\ 1
\end{pmatrix},\dots,\begin{pmatrix}
\{2,3,\dots, n \} \\1
\end{pmatrix}
 \Bigg\}$$

\noindent which clearly has no idempotent element. The result now follows.
\end{proof}

As in \cite{FOUN2}, to give the definition of the relation $\mathcal{J}^{*}$ on a semigroup $S$, we first let the $\mathcal{L}^{*}$-class containing the element $a\in S$ be denoted by $L^{*}_{a}$. Similar notation can be used for the classes of the other relations. A \emph{left} (resp., \emph{right}) $*$-\emph{ideal} of a semigroup $S$ is defined as the \emph{left} (resp., \emph{right}) ideal $I$ of $S$ such that $L^{*}_{a} \subseteq I$ (resp., $R^{*}_{a} \subseteq I$) for all $a \in  I$. A subset $I$ of $S$ is said to be a $*$-ideal of $S$ if it is both a left and a right $*$-ideal of $S$. The \emph{principal $*$-ideal} $J^{*}(a)$ generated by the element $a\in S$ is defined as the intersection of all $*$-ideals of $S$ to which $a$ belongs. The relation $\mathcal{J}^{*}$ is defined as: $a \mathcal{J}^{*} b$ if and only if $J^{*}(a) = J^{*}(b)$, where $J^{*}(a)$ is the principal $*$-ideal generated by $a$.

The upcoming lemma is critical for our further examination of the properties of $\mathcal{J}^{*}$ in $\mathcal{SS}^{\prime}_{n}$.

\begin{lemma}[\cite{FOUN2}, Lemma 1.7]\label{jj}  Let $a$  be an element of a semigroup $S$. Then $b \in J^{*}(a)$ if and only if there are elements $a_{0},a_{1},\dots, a_{n}\in  S$, $x_{1},\dots,x_{n}, y_{1}, \dots,y_{n} \in S^{1}$ such that $a = a_{0}$, $b = a_{n}$, and $(a_{i}, x_{i}a_{i-1}y_{i}) \in \mathcal{D}^{*}$ for $i = 1,\dots,n.$
\end{lemma}

As in \cite{ua}, we now have the following.

\begin{lemma}\label{jjj} For $\alpha, \, \beta \in \mathcal{SS}^{\prime}_{n}$, let $\alpha\in J^{*}(\beta)$. Then $\mid \im \, \alpha \mid\leq \mid \im \,\beta \mid$.
\end{lemma}
\begin{proof}
Let $ \alpha \in J^{*}(\beta)$. Then, according to Lemma \ref{jj}, there exist $\beta_{0}, \beta_{1},\dots, \beta_{n}$ $\in\mathcal{SS}^{\prime}_{n}$ , $\gamma_{1}, \dots, \gamma_{n}$, $\tau_{1}, \dots, \tau_{n}$ in $(\mathcal{SS}^{\prime}_{n})^{1}$ such that $\beta=\beta_{0}$,  $\alpha=\beta_{n}$, and $(\beta_{i}, \gamma_{i}\beta_{i-1}\tau_{i})\in \mathcal{D}^{*}$ for $i =1,\dots,n.$ Thus, by Lemma \ref{uaaaa}, this implies that
\[\mid\im \,\beta_{i} \mid= \mid\im \,  \gamma_{i}\beta_{i-1}\tau_{i} \mid\leq \mid\im \, \beta_{i-1} \mid ,\] \noindent so that
\[\mid \im \, \alpha \mid\leq \mid \im \,\beta \mid.\] \noindent The result now follows.
\end{proof}

\begin{lemma}\label{uaaaaa} On the  small Schr\"{o}der semigroup   $\mathcal{SS}^{\prime}_{n}$, we have $\mathcal{J}^{*}=\mathcal{D}^{*}$.
\end{lemma}
\begin{proof}We only need to show that $\mathcal{J}^{*} \subseteq \mathcal{D}^{*}$ (since $\mathcal{D}^{*} \subseteq \mathcal{J}^{*}$). Suppose $(\alpha,\beta) \in \mathcal{J}^{*}$. Then, $J^{*}(\alpha)=J^{*}(\beta)$, implying $\alpha\in J^{*}(\beta)$ and $\beta\in J^{*}(\alpha)$. By Lemma \ref{jjj}, this implies \[\mid \im \, \alpha \mid = \mid \im \, \beta \mid.\] Therefore, by Theorem \ref{a1} (iv), we have \[\mathcal{J}^{*} \subseteq \mathcal{D}^{*},\] as required.
\end{proof}

\begin{lemma}\label{un} On the semigroup $S$ in $\{\mathcal{SS}^{\prime}_{n},  \,  K(n,p) \}$, every $\mathcal{R}^{*}-$class  contains a unique idempotent.
\end{lemma}
\begin{proof} This follows from the fact that  \textbf{Ker }$\alpha$ can only admit one image subset of $[n]$ so that $\alpha$ is an idempotent element of $S$, by the decreasing property of  elements in $S$.
\end{proof}

\begin{remark}\label{remm}\begin{itemize}
             \item[(i)] It is now clear that, for each $1 \leq p \leq n$, the number of $\mathcal{R}^{*}$-classes in $J^{*}_{p} = \{\alpha\in \mathcal{SS}^{\prime}_{n}: \, |\im \, \alpha|=p\}$ is equal to the number of all possible partially ordered partitions of $[n]\setminus \{1\}$ into $p$ parts. This is equivalent to the number of $\mathcal{R}$-classes in $\{\alpha\in \mathcal{OP}_n: \, |\im \, \alpha|=p\}$ \emph{(}i.e., $\sum\limits_{r=p}^{n}{\binom{n}{r}}{\binom{r-1}{p-1}}$, as shown in \emph{[\cite{al3}, Lemma 4.1]}\emph{)}, minus the number of $\mathcal{R}$-classes in $\{\alpha\in \mathcal{SS}_n: \, |\im \, \alpha|=p\}$ \emph{(}i.e., $\sum\limits_{r=p}^{n}{\binom{n-1}{r-1}}{\binom{r-1}{p-1}}$, as shown in \emph{[\cite{zuf}, Lemma 2.20]}\emph{)}. This number is obviously equal to $\sum\limits_{r=p}^{n-1}{\binom{n-1}{r}}{\binom{r-1}{p-1}}$.
             \item[(ii)] If $S\in \{{RSS}_{n}^{\prime}(p),  \,  K(n,p) \}$ then the characterizations of the starred Green's relations in Theorem \ref{a1} also hold in $S$.
           \end{itemize}
\end{remark}

Thus, the semigroup $K(n,p)$, like $\mathcal{SS}^{\prime}_{n}$ is the union of $\mathcal{J}^{*}$ classes \[ J_{o}^{*}, \, J_{1}^{*}, \, \dots, \, J_{p}^{*},\]
where \[J_{p}^{*}=\{\alpha\in K(n,p): \, |\im \, \alpha|=p\}.\]

Furthermore, $K(n,p)$ has $\sum\limits_{r=p}^{n-1}{\binom{n-1}{r}}{\binom{r-1}{p-1}}$ $\mathcal{R}^{*}$-classes and $\binom{n}{p}$ $\mathcal{L}^{*}$-classes in each $J^{*}_{p}$. Consequently, the Rees quotient semigroup ${RLS}_{p}(n)$ has $\sum\limits_{r=p}^{n-1}{\binom{n-1}{r}}{\binom{r-1}{p-1}}+1$ $\mathcal{R}^{*}$-classes and $\binom{n}{p}+1$ $\mathcal{L}^{*}$-classes. (The term 1 is derived from the singleton class containing the zero element in every instance.)

   Before we present the next result, we first note that the number of idempotents in the large and small Schr\"{o}der monoids $\mathcal{LS}_{n}$ and $\mathcal{SS}_{n}$ is found to be $\frac{3^{n}+1}{2}$ (see [\cite{al2}, Proposition 3.5]) and $3^{n-1}$ (see [\cite{zuf}, Theorem 2.22]), respectively. Thus, we present the theorem below.
  \begin{theorem} Let $\mathcal{SS}^{\prime}_{n}$ be as defined in \eqref{qn2}. Then  $|E(\mathcal{SS}^{\prime}_{n})|=\frac{3^{n-1}+1}{2}$.
  \end{theorem}
  \begin{proof} Notice that $E(\mathcal{SS}^{\prime}_{n})\cap E(\mathcal{SS}_{n})=\emptyset$ and $E(\mathcal{SS}^{\prime}_{n})\cup E(\mathcal{SS}_{n})=E(\mathcal{LS}_{n})$. Therefore, the result follows from the fact that $|E(\mathcal{SS}^{\prime}_{n})|= |E(\mathcal{LS}_{n})\setminus E(\mathcal{SS}_{n})|=|E(\mathcal{LS}_{n})|-| E(\mathcal{SS}_{n})|$.
  \end{proof}
We  deduce the following identity.
\begin{corollary}On the semigroup $\mathcal{SS}^{\prime}_{n}$ we have $\sum\limits_{p=0}^{n-1}{\sum\limits_{r=p}^{n-1}{\binom{n-1}{r}}{\binom{r-1}{p-1}}}=\frac{3^{n-1}+1}{2}$.
\end{corollary}

\section{Rank properties}

 For a  semigroup $S$  and a   nonempty subset of $S$ say $A$, the  \emph{smallest subsemigroup} of $S$ that contains $A$ is denoted by $\langle A \rangle$ and is referred to\emph{ as subsemigroup generated by $A$}. If  $A$  is a finite subset of  $S$ such that $\langle A \rangle$ equals $S$, then $S$ is called  a \emph{finitely-generated semigroup}. The \emph{rank} of a finitely generated semigroup $S$ is defined and denoted as:
\[
\text{rank}(S) = \min\{|A| : \langle A \rangle = S\}.
\]
\noindent If the set $A$ consists only of idempotents in $S$, then $S$ is said to be an \emph{idempotent-generated semigroup} (equivalently, a \emph{semiband}), and the idempotent-rank is denoted by $\text{idrank}(S)$. For a more  detailed explanation of the rank properties of a semigroup, we refer the reader to \cite{hrb, hrb2}. There are various classes of transformation semigroups, whose ranks have been investigated, see for example \cite{g1,gu1,gm,gm3,hf,hrb,hrb2,HRS,zm1,zuf}.  In \cite{zuf}, the ranks of the Schr\"{o}der monoids $\mathcal{LS}_{n}$ and $\mathcal{SS}_{n}$, as well as the ranks of their respective certain two-sided ideals and their Rees quotients, have been investigated. In fact, these semigroups have been shown to be idempotent-generated, and their  ranks have been investigated  as well. Our aim is to compute the rank of the two-sided ideal $K(n,p)$ of the Schr\"{o}der semigroup $\mathcal{SS}^{\prime}_{n}$, thereby obtaining the rank of $\mathcal{SS}^{\prime}_{n}$ as a special case. Firstly, we note the following well-known result about decreasing maps \cite{al1, umar}.
\begin{lemma}\label{hqd} For all  order decreasing  partial maps $\alpha$ and $\beta$ on $A\subseteq [n]$, $F(\alpha\beta)=F(\alpha)\cap F(\beta)=F(\beta\alpha)$.
\end{lemma}
\begin{proof} If $\alpha$ or $\beta$ is zero (i.e., the empty map), the result follows. When $\alpha$ and $\beta$ are nonzero  order decreasing  partial maps, the proof is the same as the proof of Lemma 2. 1. 3 in \cite{umar}.
\end{proof}

We initiate our examination of the rank properties of the semigroup $S\in\{K(n,p), \mathcal{SS}^{\prime}_{n}\} $ by first introducing the following definition about injective elements in $J^{*}_{p}$ for any $1\le p\le n-1$.

\begin{definition} An injective map  $\alpha$ in $ J^{*}_{p}$  is said  to be a \emph{requisite element} if it is of the form: \[\alpha_{i}=\begin{pmatrix}2&\cdots&i& a_{i}&\cdots& a_{p}\\1&\cdots&i-1&a_i&\cdots&a_{p}\end{pmatrix},\]
  \noindent where $1<i<a_{i}<a_{i+1}<\cdots< a_{p}\leq n$.
\end{definition}
\begin{remark} \label{requ}If $\alpha_{i}$ is a requisite element, then observe that for each $1\leq i\le p$, \[\dom \, \alpha_{i}=\{2,\ldots,i, a_{i}, \dots, a_{p}\}=\{2, \ldots,i\}\cup F(\alpha_{i}).\]
\noindent Moreover, $\alpha_{i}$ is unique in  $L^{*}_{\alpha_{i}}$, in the sense that no two $\alpha_{i}$'s belong to the same $\mathcal{L}^{*}-$class. However, an $\mathcal{R}^{*}-$class may contain more than one requisite element. In fact, in $J^{*}_{n-1}$ all the $n-1$ requisite elements belong to a single $\mathcal{R}^{*}-$class.
\end{remark}

 We immediately have the following lemma.
\begin{lemma}\label{ms} If $\alpha_{i}$ is the unique requisite element in $L^{*}_{\alpha}$  in $J_{p}^{*}$ \emph{(}$1\le p\le n-1$\emph{)}, then there exists $\beta\in R^{*}_{\alpha}$ such that $\alpha=\beta\alpha_{i}$.
\end{lemma}
 \begin{proof} Let  $\alpha_{i}$ be the unique requisite element in $L^{*}_{\alpha}$ in $ J^{*}_{p}$, where $\alpha$ is  as expressed in \eqref{1}. Now since  $\alpha_{i}\in L^{*}_{\alpha}$ then $\im \, \alpha =\im \, \alpha_{i}$, and so $a_{j}=j$ for all $1\leq j\leq i-1$, so that  $\alpha$ and $\alpha_{i}$ are:
 \[\alpha=\begin{pmatrix}A_{1}&\cdots&A_{i-1}& A_{i}&\cdots& A_{p}\\1&\cdots&i-1&a_i&\cdots&a_{p}\end{pmatrix} \text{ and }\alpha_{i}=\begin{pmatrix}2&\cdots&i& a_{i}&\cdots& a_{p}\\1&\cdots&i-1&a_i&\cdots&a_{p}\end{pmatrix},\]\noindent respectively.
By order preservedness property, it is clear that $\min A_{j}<\min A_{j+1}$ for all $1\le j\le i-1$, and since $2\leq \min A_{1}$, it follows that $j+1\leq \min A_{j}$. Thus,  the map $\beta$ defined as:
\[\beta=\begin{pmatrix}A_{1}&\cdots&A_{i-1}& A_{i}&\cdots& A_{p}\\2&\cdots&i&a_i&\cdots&a_{p}\end{pmatrix}\in \mathcal{SS}^{\prime}_{n}.\]
\noindent  Furthermore, notice that $\ker \, \alpha=\ker \, \beta$, and so $\beta\in R^{*}_{\alpha}$. Clearly, $\beta\alpha_{i}=\alpha$.

The proof is now complete.
\end{proof}


\begin{theorem}\label{hq} Let $\alpha\in \mathcal{SS}^{\prime}_{n}$ be as expressed in \eqref{1}. Then \begin{itemize}
                                                                            \item[(i)] if $a_{1}\neq 1$, then $\alpha$ is idempotent-generated;
                                                                            \item[(ii)] if $a_{1}= 1$, then $\alpha$ is a product of idempotents and the unique requisite element in $L^{*}_{\alpha}$.
                                                                          \end{itemize}
\end{theorem}
\begin{proof} Let $\alpha\in \mathcal{SS}^{\prime}_{n}$ be as expressed in  \eqref{1}.

\noindent\textbf{(i.)} Suppose $a_{1}\neq 1$ and let $U=\{\alpha\in\mathcal{SS}^{\prime}_{n}: \, 1\not\in \, \im \, \alpha \}$. Then it is not difficult to see that $U$ is a subsemigroup of $\mathcal{SS}^{\prime}_{n}$, which is isomorphic to $\mathcal{LS}_{n-1}$. Hence by [\cite{zuf}, Lemma 3.2], $U$ is idempotent generated.

 \noindent\textbf{ (ii.)} Now suppose $a_{1}=1$. Thus, by Lemma \ref{ms} $\alpha$ can be expressed as  \[\alpha=\beta\alpha_{i},\] \noindent for some  $\beta\in R^{*}_{\alpha}$, where $\alpha_{i}$ is the unique requisite  element in $L^{*}_{\alpha}$. To be precise,
 \[\beta=\begin{pmatrix}A_{1}&\cdots&A_{i-1}& A_{i}&\cdots& A_{p}\\2&\cdots&i&a_i&\cdots&a_{p}\end{pmatrix} \text{ and }\alpha_{i}=\begin{pmatrix}2&\cdots&i& a_{i}&\cdots& a_{p}\\1&\cdots&i-1&a_i&\cdots&a_{p}\end{pmatrix}.\]
\noindent Notice that in the $\im \, \beta$,  $a_{1}=2\neq 1$, and so $\beta\in U$. Thus, by (i) $\beta$ is idempotent-generated, as stipulated.
\end{proof}

 Consequently, we have the following corollary.
\begin{corollary} The semigroup $\mathcal{SS}^{\prime}_{n}$ is generated by its idempotents and  requisite elements.
\end{corollary}
Notice that in the proof of Theorem \ref{hq}, $|\im \, \alpha|=h(\alpha)=h(\beta)=h(\alpha_{i})=p$ for all $1\le i\le p$. Based on this, we have the following result.

\begin{lemma}\label{hh}
Every element  in $\mathcal{SS}^{\prime}_{n}$  of height $p$ can be expressed as a product of idempotents and  requisite elements in $\mathcal{SS}^{\prime}_{n}$, each of height $p$.
\end{lemma}

 Now let $M(p)=\{\alpha\in {RSS}_{n}^{\prime}(p): \alpha \text{ is a requisite element}\}$ and  $E({RSS}_{n}^{\prime}(p)\setminus\{0\})$ be  the collection of all nonzero idempotents in $ {RSS}_{n}^{\prime}(p)$. Then we have the following lemma.

  \begin{lemma}\label{nid}  For $1\le p\leq n-1$, we have \begin{itemize} \item[(i)] $|M(p)|=\binom{n-1}{p-1}$; \item[(ii)] $|E({RSS}_{n}^{\prime}(p))\setminus\{0\}|=\sum\limits_{r=p}^{n-1}{\binom{n-1}{r}}{\binom{r-1}{p-1}}$.\end{itemize}
\end{lemma}
\begin{proof}
\noindent \textbf{(i.)} Given that $1$ must be included in every image set that contains $p$ images, we can then choose the remaining $p-1$ images from the available $n-1$ elements in $\binom{n-1}{p-1}$ distinct ways, as stipulated.

\noindent\textbf{(ii.)} The result follows from Remark \ref{remm}.
\end{proof}

Now, let $G(p)=M(p)\cup (E({RSS}_{n}^{\prime}(p))\setminus\{0\})$. The next result shows that the subset $G(p)$ of ${RSS}_{n}^{\prime}(p)$  is the minimum generating set of ${RSS}_{n}^{\prime}(p)\setminus \{0\}$.
 \begin{lemma}\label{minnn} Let $\alpha$, $\beta$ be elements in ${RSS}_{n}^{\prime}(p)\setminus \{0\}$. Then $\alpha\beta\in G(p)$  if and only if $\alpha, \, \beta\in G(p)$  and $\alpha\beta=\alpha$ or $\alpha\beta=\beta$.
\end{lemma}

\begin{proof} Suppose $\alpha\beta\in G(p)$. Thus, either $\alpha\beta\in E({RSS}_{n}^{\prime}(p)\setminus\{0\})$ or $\alpha\beta\in M(p).$ We consider the two cases separately.

\noindent \textbf{Case i.} Suppose $\alpha\beta\in E({RSS}_{n}^{\prime}(p)\setminus\{0\})$.
 Then \[
p= f(\alpha\beta)\leq f(\alpha)\leq |\im \, \alpha|=p,\]
\[
p= f(\alpha\beta)\leq f(\beta)\leq |\im \, \beta|=p.\]

This ensures that \[F(\alpha)=F(\alpha\beta)=F(\beta), \] \noindent and so $\alpha, \, \beta\in E({RSS}_{n}^{\prime}(p)\setminus\{0\})\subset G(p)$, which implies that $\alpha, \, \beta\in G(p)$  and $\alpha\beta=\alpha$.

\noindent \textbf{Case ii.} Now suppose $\alpha\beta\in M(p)$. Thus $\alpha\beta$ is a requisite element which has the form \[\alpha\beta=\begin{pmatrix}2& \cdots &i&a_{i}&\cdots&a_p\\1&\cdots&i-1&a_{i}&\cdots&a_p\end{pmatrix},\] \noindent where $1<i<a_{i}<a_{i+1}<\cdots< a_{p}\leq n$. This means that $\dom \, \alpha=\dom \, \alpha\beta$ and $\im \, \beta =\im \, \alpha\beta$. Thus, \[\alpha=\begin{pmatrix}2& \cdots &i&a_{i}&\cdots&a_p\\2\alpha&\cdots&i\alpha&a_{i}&\cdots&a_p\end{pmatrix} \text{ and } \beta=\begin{pmatrix}2\alpha& \cdots &i\alpha&a_{i}&\cdots&a_p\\1&\cdots&i-1&a_{i}&\cdots&a_p\end{pmatrix},\] since $\im \, \alpha=\dom \, \beta$.
\noindent The claim here is that $\alpha$ must be an idempotent. Notice that either $2\alpha=1$ or $2\alpha=2$. In the former, we see that $1=2\alpha\in \dom \, \beta$ which is a contradiction. Therefore $2\alpha=2$ which implies $j\alpha=j$ for all $2\leq j\leq i$. Thus $\alpha$ is an idempotent, that is, \[\alpha=\begin{pmatrix}2& \cdots &i&a_{i}&\cdots&a_p\\2&\cdots&i&a_{i}&\cdots&a_p\end{pmatrix} \text{ and }\beta=\begin{pmatrix}2& \cdots &i&a_{i}&\cdots&a_p\\1&\cdots&i-1&a_{i}&\cdots&a_p\end{pmatrix}.\] \noindent  Therefore, $\beta\in M(p)\subset G(p)$ and $\alpha\in E({RSS}_{n}^{\prime}(p)\setminus\{0\})\subset G(p)$ and also $\alpha\beta=\beta.$
The converse is obvious.
\end{proof}

We have now established a key result of this section.

\begin{theorem}\label{pb} Let ${RSS}_{n}^{\prime}(p)$ be as defined in \eqref{knn}. Then  \[\text{rank } {RSS}_{n}^{\prime}(p) =\binom{n-1}{p-1}+\sum\limits_{r=p}^{n-1}{\binom{n-1}{r}}{\binom{r-1}{p-1}}.\]
\end{theorem}
\begin{proof} The proof follows from Lemmas \ref{minnn} \& \ref{nid}.
\end{proof}

The next lemma is crucial in determining  the ranks of the Schr\"{o}der semigroup  $\mathcal{SS}^{\prime}_{n}$ and its two sided ideal $K(n,p)$.  Now  for $1\leq p\leq n-1$ let \[J^{*}_{p}=\{\alpha\in \mathcal{SS}^{\prime}_{n}: |\im \, \alpha|=p \}.\] Moreover, for $1\leq p\leq n-1$ let $M(p)$ be the collection of all requisite elements in $J_{p}^{*}$, and let \[G(p)=M(p)\cup E(J^{*}_{p}).\]\noindent Then we have the following lemmas.
 \begin{lemma} For $1\le p\le n-1$, $|G(p)|=\binom{n-1}{p-1}+\sum\limits_{r=p}^{n-1}{\binom{n-1}{r}}{\binom{r-1}{p-1}}$.
\end{lemma}
\begin{proof} The result follows from Lemma \ref{nid}.
\end{proof}
\begin{lemma}\label{lm1} For $0\leq p\leq n-3$, $J^{*}_{p}\subset \langle J^{*}_{p+1}\rangle$. In other words, if $\alpha\in J^{*}_{p}$ then $\alpha\in \langle J^{*}_{p+1}\rangle$ for $1\leq p\leq n-3$.
\end{lemma}
\begin{proof} Using Theorem \ref{hq}, it suffices to prove that every element in $G(p)$ can be expressed as a product of elements in $G(p+1)$. That is to say  every idempotent of height $p$ can be expressed as product of idempotents of height $p+1$, and every requisite element, can be expressed as product of requisite elements  of height $p+1$. Thus, we consider the elements of $E(J^{*}_{p})$  and $M(p)$ separately.\\

\noindent\textbf{i.} The  elements in $E(J^{*}_{p})$:\\

 Let $\epsilon\in E(J^{*}_{p})$ be expressed as: \[\epsilon= \begin{pmatrix}
A_1 &  \cdots  & A_p \\
t_1 &  \cdots  & t_p
\end{pmatrix},  \] \noindent where $t_{i}=\min A_{i}$ for all $1\le i\le p$. Notice that  since $1\not\in \dom \, \epsilon$, then $1\not\in A_{1}$, and so $t_{1}\neq 1$. This implies that  $\epsilon\in E(U)$, where $U=\{\alpha\in \mathcal{SS}^{\prime}_{n}: \, 1\not\in \, \im \, \alpha \}$, which is isomorphic to the large Schr\"{o}der monoid  $\mathcal{LS}_{n-1}$, and so, the proof follows from [\cite{zuf}, Lemma 3.7].\\

\noindent\textbf{ii.} The  elements in $M(p)$:\\

Let $\alpha_{i}$ be a requisite element of height $p$, which has the form: \[\alpha_{i}= \begin{pmatrix}2& \cdots &i&a_{i}&\cdots&a_p\\1&\cdots&i-1&a_{i}&\cdots&a_p\end{pmatrix},  \]\noindent where  $1<i<a_{i}<a_{i+1}<\cdots< a_{p}\leq n$.
Now, since  $p\leq n-3$, it implies that $(\dom \, \alpha_{i}\cup \im \, \alpha_{i})^{c}$ contains at least two elements, specifically $c$ and $d$. We may suppose without loss of generality that, $d<c$.
Now let $A=\dom \, \alpha_{i}\cup\{d\}$ and $B=\dom \, \alpha_{i}\cup\{c\}$. Now  define $\beta$ and $\gamma$ as follows:

For $x\in A$ and $y\in B$
\[x\beta=\left\{
                                                                                                                                \begin{array}{ll}
                                                                                                                                  x, & \hbox{if $x\neq d $;} \\
                                                                                                                                  d, & \hbox{if $x=d$}
                                                                                                                                \end{array}
                                                                                                                              \right.,
\] \noindent and
\[y\gamma=\left\{
                                                                                                                                \begin{array}{ll}
                                                                                                                                  y\alpha_{i}, & \hbox{if $y\neq c$;} \\
                                                                                                                                  c, & \hbox{if $y=c$}
                                                                                                                                \end{array}
                                                                                                                              \right.
.\]
\noindent Notice that $\beta\in E(J^{*}_{p+1})\subset G(p+1)$, and it is not difficult to see that  $\gamma$ is a requisite element in $M(p+1)\subset G(p+1)$. One can now easily show that $\alpha_{i}=\beta\gamma$.

The proof of the lemma is now complete.
\end{proof}

Consequently, we have the following result.

\begin{theorem}\label{knp} Let $K(n,p)$ be as defined in \eqref{kn}. Then  for $1\leq p\leq n-2$, we have the \[\text{rank } K(n,p) =\binom{n-1}{p-1}+\sum\limits_{r=p}^{n-1}{\binom{n-1}{r}}{\binom{r-1}{p-1}}.\]
\end{theorem}
\begin{proof}Notice that by Lemma \ref{lm1},   $\langle J^{*}_{p} \rangle= K(n,p)$ for all $1\le p\leq n-2$.  Notice also that, $\langle E({RSS}_{n}^{\prime}(p)\setminus\{0\})\cup M(p)\rangle= J^{*}_{p}$. The result now follows from Theorem \ref{pb}.
\end{proof}

It is important to note that Lemma \ref{lm1} does not cover the case $p=n-2$, meaning that the assertion  $J^{*}_{n-2}\subset \langle J^{*}_{n-1}\rangle$ is not true. This can be demonstrated  by first noting that every element $\alpha$ in $J^{*}_{n-1}$ is an injective  map, because $\dom \, \alpha=[n]\setminus\{1\}$. However, in $J^{*}_{n-2}$ there are maps which are not necessarily injective. For example, a map $\alpha$ with $\dom \, \alpha=[n]\setminus\{1\}$ and $|\im \, \alpha|=n-2$,  and cannot be generated by  injective maps, since injective maps alone can only generate  injective maps. Therefore, we need to investigate the generating set of  $J^{*}_{n-1}$, and its relationship with $G(n-2)$.

The elements of $J^{*}_{n-1}$ are obviously injective, isotone and decreasing maps from $[n]\setminus\{1\}$ into $[n]$, and are of two types:
requisite elements $\alpha_{i}$ of the form \begin{equation}\label{kk-1}\alpha_{i}=\left(\begin{array}{cccccccc}
                                                                            {2}& 3& \cdots &i-1 &i& i+1&\cdots&  n \\
                                                                            {1} & 2&\cdots & i-2 & i-1 & i+1&\cdots&   n
                                                                          \end{array}
\right)  ~~(2\leq i\leq n),\end{equation}
\noindent and the unique idempotent  \[\epsilon=\left(\begin{array}{cccccccc}
                                                                            {2}&  \cdots &  n \\
                                                                            {2} & \cdots &   n
                                                                          \end{array}
\right),\]
\noindent which is the unique left identity on $\mathcal{SS}^{\prime}_{n}$, but not an identity. Perhaps we may call the semigroup $\mathcal{SS}^{\prime}_{n}$, a\emph{ left monoid.} It is  also clear that $J^{*}_{n-1}$ has only 1 $\mathcal{R}^{*}-$class and $n$ $\mathcal{L}^{*}-$classes. Obviously, there are $n-1$ $\mathcal{L}^{*}-$classes each containing   a unique requisite element. Thus, $G(n-1)=M(n-1)\cup\{\epsilon\}=J^{*}_{n-1}$, and so $|G(n-1)|=n$.  We now have the following:
\begin{lemma}\label{minnnn} Let $\alpha$, $\beta$ be elements in $\langle J^{*}_{n-1}\rangle$. Then $\alpha\beta\in G(n-1)$  if and only if $\alpha, \, \beta\in G(n-1)$  and $\alpha\beta=\alpha$ or $\alpha\beta=\beta$.
\end{lemma}
\begin{proof} The proof is similar to that of Lemma \ref{minnn}.
\end{proof}
\begin{lemma}\label{jnn} The rank$\langle J^{*}_{n-1} \rangle=n$.
\end{lemma}
\begin{proof} The result follows from Lemma \ref{minnnn} and the fact that $|G(n-1)|=|J^{*}_{n-1}|$.
\end{proof}

 The  elements of $M(n-2)$ are generally of the form: \begin{equation}\label{iki}{\alpha}_{i, k}=\left(\begin{array}{ccccccccc}
                                                                            {2}& \cdots &i &i+1& \cdots& k-1 &   k+1 & \cdots& n \\
                                                                            {1} & \cdots & i-1 & i+1 &\cdots& k-1  &k+1 &\cdots & n
                                                                          \end{array}
\right)\text{ }(2\leq i< k\leq n).\end{equation}\noindent Notice that $1, \, k\not\in \dom \, {\alpha}_{ i, k}$ and $i,k\not\in \im \, {\alpha}_{k, i}$. Also, the element

\begin{equation}\label{i1i}{\epsilon}_{ 1,k}=\left(\begin{array}{ccccccc}
                                                                            {2}&\cdots &k-1 &k+1& \cdots& n \\
                                                                            {2} & \cdots & k-1 & k+1 &\cdots& n
                                                                          \end{array}
\right)\end{equation} \noindent is a partial identity in $G(n-2)$ for all $2\leq k\leq n$. Notice also  that $1, \, k\not\in \dom \, {\epsilon}_{  1, k}$ and $1,k\not\in \im \, {\epsilon}_{  1, k}$.

The subsequent lemma demonstrates that the requisite elements together with the partial identity $\epsilon_{1,2}$ in $G(n-2)$ are not required to be included in any minimal  generating set of $\langle J^{*}_{n-2} \cup J^{*}_{n-1} \rangle$. However, the requisite elements in $G(n-1)$ suffice to be included in any  generating set of $\langle J^{*}_{n-2} \cup J^{*}_{n-1} \rangle$.
\begin{lemma}\label{ss1} $\langle G(n-2) \rangle \subseteq \langle \left(G(n-2)\setminus (M(n-2)\cup\{\epsilon_{1,2}\})\right)\cup G(n-1) \rangle=\langle J^{*}_{n-2}\cup J^{*}_{n-1}  \rangle=\mathcal{SS}^{\prime}_{n}$.
\end{lemma}

\begin{proof} It is enough to show that $M(n-2)\cup\{\epsilon_{1,2}\}\subset \langle G(n-1)\cup (E(J^{*}_{n-2})\setminus\{\epsilon_{1,2}\}) \rangle$.
 Now  let   ${\alpha}_{  i, k} \in M(n-2)$ be  as expressed in \eqref{iki}. So, take the idempotent  ${\epsilon}_{1,k}\in G(n-2)$ and   $\alpha_{i}\in G(n-1)$  as expressed in \eqref{i1i} \& \eqref{kk-1}, respectively; and observe that for any $i<k$, we have: \begin{align*}{\epsilon}_{ 1,k}\alpha_{i}&=\left(\begin{array}{ccccccccccc}
                                                                            {2}&\cdots &i&i+1&\cdots&k-1 &k+1& \cdots& n \\
                                                                            {2} & \cdots &i&i+1&\cdots& k-1 & k+1 &\cdots& n
                                                                          \end{array}
\right)\left(\begin{array}{cccccccc}
                                                                            {2}& 3& \cdots  &i& i+1&\cdots&  n \\
                                                                            {1} & 2&\cdots & i-1 & i+1&\cdots&   n
                                                                          \end{array}
\right)\\&=\left(\begin{array}{ccccccccccc}
                                                                            {2}&\cdots& i&i+1&\cdots&k-1 &k+1& \cdots& n \\
                                                                            {1} & \cdots &i-1&i+1&\cdots& k-1 & k+1 &\cdots& n
                                                                          \end{array}
\right)\\&={\alpha}_{ i,k}.
\end{align*}

 Similarly, observe that if $i=2$, then \begin{align*}\alpha_{2}^{2}&=\left(\begin{array}{cccc}
                                                                            {2}&3 & \cdots& n \\
                                                                            {1} & 3 &\cdots& n
                                                                          \end{array}
\right)\left(\begin{array}{cccc}
                                                                            {2}&3 & \cdots& n \\
                                                                            {1} & 3 &\cdots& n
                                                                          \end{array}
\right)\\&=\left(\begin{array}{ccc}
                                                                            {3}&\cdots& n \\
                                                                            {3} & \cdots & n
                                                                          \end{array}
\right)\\&={\epsilon}_{ 1,2}.
\end{align*}
\noindent The result now follows.
\end{proof}

Finally, we conclude the paper with the following result:

\begin{theorem} Let $\mathcal{SS}_{n}^{\prime}$ be as defined in \eqref{qn2}. Then the rank $\mathcal{SS}_{n}^{\prime}=3n-4$.
\end{theorem}
\begin{proof} Clearly $(G(n-2)\setminus (M(n-2)\cup\{\epsilon_{1,2}\}))\cup G(n-1)$ is the minimum generating set of $\langle J^{*}_{n-2}\cup J^{*}_{n-1} \rangle=\mathcal{SS}_{n}^{\prime}$, and so by Lemmas  \ref{jnn}  \& \ref{nid} we see that
\begin{align*} \text{ rank }\mathcal{SS}_{n}^{\prime}&= |G(n-2)|- |M(n-2)|-1+|G(n-1)|\\&=\binom{n-1}{n-2}+ \sum\limits_{r=n-2}^{n-1}{\binom{n-1}{r}}{\binom{r-1}{n-3}}- \binom{n-1}{n-2}-1+n\\&=\sum\limits_{r=n-2}^{n-1}{\binom{n-1}{r}}{\binom{r-1}{n-3}}+n-1=3n-4,
\end{align*}
as stipulated.
\end{proof}

\noindent{\bf Acknowledgements, Funding and/or Conflicts of interests/Competing interests.} The first named author would like to thank Bayero University and TETFund (TETF/ES/UNI/KANO/TSAS/2022) for financial support. He would also like to thank Sultan Qaboos University, Oman,  for hospitality during a 1-year postdoc research visit to the institution from June 2024 to December 2024.

\end{document}